\documentclass[11pt]{amsart}
\usepackage[dvips]{graphicx}
\usepackage{amsmath,graphics}
\usepackage{amsfonts,amssymb}
\usepackage{xypic}
\usepackage{comment}
\usepackage{color}
\usepackage{stmaryrd}
\usepackage{enumerate}
\specialcomment{proofc}{}{}
\includecomment{proofc}
\theoremstyle{plain}
\newtheorem*{theorem*}{Theorem}
\newtheorem*{lemma*} {Lemma}
\newtheorem*{corollary*} {Corollary}
\newtheorem*{proposition*}{Proposition}
\newtheorem*{conjecture*}{Conjecture}
\newtheorem{theorem}{Theorem}[section]
\newtheorem{lemma}[theorem]{Lemma}
\newtheorem*{theorem1*}{Theorem 1}
\newtheorem*{theorem2*}{Theorem 2}
\newtheorem*{theorem3*}{Theorem 3}
\newtheorem{corollary}[theorem]{Corollary}
\newtheorem{proposition}[theorem]{Proposition}

\theoremstyle{remark}
\newtheorem*{remark}{Remark}
\newtheorem*{remarks}{Remarks}

\newtheorem{example*}{Example}

\newtheorem*{ack}{Acknowledgement}
\theoremstyle{definition}

   \def\Z{\Bbb{Z}}  
    
    \def\bp{\begin{pmatrix}}
 \def\ep{\end{pmatrix}} \def\bn{\begin{enumerate}} 
   \def\en{\end{enumerate}}
\def\ba{\begin{array}} \def\ea{\end{array}}

\def\be{\begin{equation}} \def\ee{\end{equation}}

\begin{document}
\title{A note on the fundamental group of Kodaira fibrations}

\author{Stefano Vidussi}
\address{Department of Mathematics, University of California,
Riverside, CA 92521, USA} \email{svidussi@ucr.edu}
\date{\today}

\begin{abstract} The fundamental group $\pi$  of a Kodaira fibration is, by definition, the extension of a surface group $\Pi_b$ by another surface group $\Pi_g$, i.e. \[ 1 \rightarrow \Pi_g  \rightarrow \pi \rightarrow \Pi_b \rightarrow 1. \] Conversely, in \cite[Question 16] {Ca17} the author inquires about what conditions need to be satisfied by a group of that sort in order to be the fundamental group of a Kodaira fibration. In this short note we collect some restriction on the image of the classifying map $m \colon \Pi_b \to \Gamma_g$ in terms of the coinvariant homology of $\Pi_g$. In particular, we observe that if $\pi$ is  the fundamental group of a Kodaira fibration with \textit{relative irregularity} $g-s$, then $g \leq 1+ 6s$, and we show that this effectively constrains the possible choices for $\pi$, namely that there are group extensions as above that fail to satisfy this bound, hence cannot be the fundamental group of a Kodaira fibration. A noteworthy consequence of this construction is that it provides examples of symplectic $4$--manifolds that fail to admit a K\"ahler structure for reasons that eschew the usual obstructions.

\end{abstract}

\maketitle

In this note we use the term \textit{Kodaira fibration} to refer to an algebraic surface $X$ endowed with the structure of smooth fiber bundle with bundle map $f \colon X \to \Sigma_b$ over a Riemann surface $\Sigma_{b}$, with fiber a Riemann surface $\Sigma_{g}$ so that the bundle map  is not isotrivial as a fibration (in the sense of algebraic geometry, i.e. as a pencil of curves); in particular $f$ is not  a locally trivial \textit{complex analytic} fiber bundle. It is well--known that the base and fiber genera satisfy the conditions $b \geq 2$ and $g \geq 3$, see e.g. \cite[Section V.14]{BHPV04}; moreover, by Arakelov's inequality, the signature of the intersection form of $X$ is strictly positive, see e.g. \cite[Corollary 42]{Ca17}. 

From the definition we deduce that $X$ is aspherical and, denoting $\pi := \pi_1(X)$ the fundamental group of $X$, we have a short exact sequence \begin{equation} \label{eq:ses}  1 \rightarrow \Pi_g  \rightarrow \pi \rightarrow \Pi_b \rightarrow 1 \end{equation}  where $\Pi_h$  is the fundamental group of a Riemann surface $\Sigma_h$. Namely, $\pi$ is the extension of a surface group by a surface group. In \cite[Question 16] {Ca17} the author inquires about what conditions need to be satisfied by a group of that sort in order to be the fundamental group of a Kodaira fibration, in terms of the \textit{monodromy representation} $m \colon \Pi_b \to \Gamma_g$, where $\Gamma_g$ is the genus--$g$ mapping class group, that classifies the bundle $f \colon X \to \Sigma_b$. That was not the very first time that a similar question has been raised; previous instances of this include MathOverflow posts by I. Rivin and J. Bryan, implicitly \cite{Ba12} (prompted by Bryan's post), \cite{Hi15}, and older ones, possibly starting with \cite{Wa79,Jo86}. We make  reference to \cite{Ca17} as it is there that the question was framed in the form we discuss here.
Further results on this question appeared recently  in  \cite{Ar17} and \cite{Fl17}.

As is well known, by Parshin Rigidity Theorem, for any fixed pair $(g,b)$ there can exist only finitely many extensions as in (\ref{eq:ses}) that arise as fundamental group of some Kodaira fibration, so \textit{a priori} ``most" choices of $m$ fail that condition. It is not clear, however, how to spell out explicit obstructions.  

The question of \cite{Ca17}  is asked under the obvious assumption that the image of $m$ is infinite, so that the signature does not  automatically vanish. 
To that, we can add the assumption that the first Betti number of $\pi$ (and all its finite index subgroups) must be even, lest $\pi$ fails to be K\"ahler. This further restricts, although it does not decide, possible choices for a group $\pi$ as in (\ref{eq:ses}); examples of this appear in \cite{Jo86,Ba12}.

The purpose of this note is to show how subtler restrictions follow from Xiao's constraints (\cite{Xi87}) on the relative irregularity of holomorphic fibrations. 

To start, assume that $X$ is an aspherical surface with fundamental group $\pi$ as in (\ref{eq:ses}). As discussed by Kotschick and Hillman in \cite[Proposition 1]{Ko99} and \cite[Theorem 1]{Hi00} respectively, the existence of an epimorphism from $\pi$  onto an hyperbolic surface group $\Pi_b$ with a finitely generated kernel entails the existence of a holomorphic fibration $f : X \to \Sigma_{b}$ whose fibers are smooth and of multiplicity one. Namely  $f : X \to \Sigma_{b}$ is a smooth surface bundle over a surface. By the long exact sequence of the fibration, the fiber must have genus equal to $g$. 

The first, trivial, observation is that the condition on the nonvanishing of the signature implies that the image of $m$ cannot be contained in subgroups of $\Gamma_g$ for which the signature of the corresponding surface bundle is zero; this entail the well--known facts that the image of $m$ cannot be contained in the hyperelliptic subgroup, or in the Torelli subgroup, of the mapping class group. 

We pause to give a different proof of the latter statement, that positions us  in the spirit of this note.

\begin{lemma} \label{lemma:tor} (Folklore) Let $X$ be an aspherical K\"ahler surface with fundamental group $\pi := \pi_{1}(X)$ as in (\ref{eq:ses}). Then if the image of the monodromy representation $m \colon \Pi_b \to \Gamma_g$ is contained in the Torelli subgroup, then it is trivial and $X = \Sigma_{g} \times \Sigma_{b}$. \end{lemma}  

\begin{proof} The  Hochschild--Serre spectral sequence associated to the sequence in (\ref{eq:ses}) gives, in low degrees, the long exact sequence 
\begin{equation} \label{eq:hs} H_2(X;\Z) \stackrel{f_*}{\rightarrow} H_2(\Sigma_b;\Z) \rightarrow H_1(\Sigma_g;\Z)_{\Pi_b} \rightarrow H_1(X;\Z) \stackrel{f_*}{\rightarrow} H_1(\Sigma_b;\Z) \rightarrow 0, \end{equation}  where $H_1(\Sigma_g;\Z)_{\Pi_b}$ is the group of coinvariants under the action of $\Pi_b$ on the homology of the fiber, i.e. the \textit{homological} monodromy.
It is well--known that $f_* \colon H_2(X;\Z) \rightarrow H_2(\Sigma_b;\Z)$ is an epimorphism over rational coefficients. It follows that the \textit{relative irregularity} $q_f =  q(X) - q(\Sigma_b)$  is given by the half of the rank of the group of coinvariants (which must be even for $X$ to have even first Betti number), i.e. \[ q_f =  \frac{1}{2}\textit{rk}_{\Z}H_1(\Sigma_g;\Z)_{\Pi_b}.\] As $m(\Pi_b)$ is contained in the Torelli group the action of $\Pi_b$ on $H_1(\Sigma_g;\Z)$ is trivial, hence $q_f = g$. By a result of Beauville (see \cite[Lemme]{Be82}), this can occurs if and only if $X$ is birational to the product $\Sigma_g \times \Sigma_h$ with the fibration corresponding to the projection; as we assume that $X$ is aspherical, this implies that $X$ is actually isomorphic to  $\Sigma_g \times \Sigma_h$. \end{proof}

As a Kodaira fibration is aspherical and non--isotrivial it cannot be isomorphic to a product. The Lemma implies therefore that the image of $m$  cannot be contained in the Torelli subgroup.

The take--home point of the previous proof is that we can use constraints on the relative irregularity of fibered surfaces  (as we used Beauville's), to control the coinvariants of the homological monodromy. This leads to the following.  

\begin{proposition} \label{proposition:xi} Let $X$ be an aspherical K\"ahler surface with fundamental group $\pi := \pi_{1}(X)$ as in (\ref{eq:ses}). Then if the coinvariant homology  $H_1(\Sigma_g;\Z)_{\Pi_b}$ has rank $2g-2s$, then $g \leq 1 + 6s$.  \end{proposition}
\begin{proof} Following \textit{verbatim} the proof of Lemma \ref{lemma:tor}, we have a holomorphic fibration $f \colon X \to \Sigma_{b}$ that is a smooth surface bundle over a surface,  whose low--degree homology is described in (\ref{eq:hs}). As the action of $\Pi_b$ on $H_1(\Sigma_g;\Z)$ is nontrivial, the fibration is necessarily nontrivial, hence by  \cite[Corollary 3]{Xi87} its relative irregularity satisfies the bound $q_f \leq \frac{5g+1}{6}$. By our assumption again we have $q_f = g-s$, whence the bound follows.
\end{proof}
\begin{corollary} A Kodaira fibration of  relative irregularity $q_f = g-1$ has fiber of genus $g \leq 7$. \end{corollary} 
It has been conjectured (see e.g. \cite{BG--AN18}) that Xiao's bound on then relative irregularity could be improved to $q_f \leq \frac{g}{2} + 1$ (the \textit{Modified Xiao's conjecture};  \cite{BG--AN18} makes some steps in that direction using the Clifford index of a fibration). It that were the case, we would get that the fiber genus is $4$ at most (and similar improvements would occur to Proposition \ref{proposition:xi}). As Kodaira fibrations have fibers of genus $g$ at least $3$, this would leave an even narrower room for the genera of possible fibers. In fact
with extra assumptions on the genus of the base, we can actually improve that result: in fact, as observed in \cite{Ca17}, using standard inequalities for nonisotrivial smooth fibrations (see \cite{Be82,Liu96}) it is not hard to prove that the holomorphic Euler characteristic of a Kodaira fibration satisfies the inequalities \begin{equation} \label{eq:bounds} 3(b-1)(g-1) < 3 \chi(X) < 4(b-1)(g-1) \end{equation}
so when $b=2$ there are no Kodaira fibrations of genus $3$ or $4$ (regardless of the relative irregularity). We can then rephrase the above results: 
\begin{corollary} \label{cor:conditional} If the Modified Xiao's conjecture holds true, a Kodaira fibration with fiber of genus $g$ and base of genus $b = 2$ has relative irregularity $q_f \leq g - 2$. \end{corollary} 

\begin{remark}
Without entering into much detail, we note that the inequalities in (\ref{eq:bounds}), or other standard inequalities for surfaces, don't seem to provide any further information for the case of $b=2$ beyond those observed above; for example, when $g=5$, we cannot exclude the existence of a Kodaira fibration with $\chi(X) = 5$ which may have $q_f = 4$ (although it is worth noticing that such surface would have $K_X^2 = 44$, hence slope $K_X^2/e(X) = 2.75$, higher than any known examples).  In that sense, it seems necessary to invoke the Modified Xiao's conjecture in the statement of Corollary \ref{cor:conditional}. Similarly, if $b=3$, for $g=3$ we are aware of no constraints to the existence of a Kodaira fibration with $q_f = 2$ that would have again $\chi(X) = 5$,  $K_X^2 = 44$.  The information of  (\ref{eq:bounds}) gets further diluted as $b,g$ grow. 
 \end{remark}

To complete the picture, we need to show that Proposition \ref{proposition:xi} is effective, namely that there do exist surface bundles over a surface that violate the constraints on the fiber genus therein contained. As previously observed, this entails that there exist actual groups $\pi$ as in (\ref{eq:ses}) with  coinvariant homology of rank $2g-2s$ that cannot be the fundamental groups of a Kodaira surface. 

We will provide examples for some choice of the value $s$; it is not too hard, with minor modifications of the constructions described below, to extend the class of such $s$: we leave this task to the interested reader. For good measure, the surface bundles we will describe have even Betti numbers and strictly positive signature, so there aren't (to the best of our understanding) simpler obstructions to the existence of a Kodaira fibration.

Instead of reinventing the wheel, our construction will follow the template of \cite{Ba12} (to which we refer the reader), and exploit as building block a remarkable surface bundle over a surface discovered in \cite{EKKOS02}. This manifold, that will be denoted as $Z$, has fibers of genus $3$, base of genus $9$, strictly positive signature $\sigma(Z) = 4$, and a section with self--intersection $0$. 

\begin{proposition} \label{prop:essezero} There exists an integer $s_0 \in [1,3]$ such that for every $b \geq 9$ there exists a surface bundle $Z_{g,b}$ with fiber of genus $g$ over a surface of genus $b$, of strictly positive signature, whose coinvariant homology $H_1(\Sigma_g;\Z)_{\Pi_b}$ has rank $2g-2s_0$ and $g > 1+ 6s_0$.  Its fundamental group cannot occur as fundamental group of a Kodaira fibration. \end{proposition}
\begin{proof} We start by producing the example with $b = 9$. Consider the aforementioned surface bundle $\varphi \colon Z \to \Sigma_9$; as the fibers have genus $3$, we must have $\textit{rk}_{\Z}H_1(\Sigma_3;\Z)_{\Pi_9}  \in [0,5]$. (It is probably not too difficult to actually determine the exact value -- which must be strictly smaller than the maximal possible value $\textit{rk}_{\Z}H_1(\Sigma_3;\Z) = 6$, otherwise the monodromy would be contained in the Torelli group -- but our construction is independent of the outcome of such calculation; as this construction can be applied to surface bundles other than $Z$, we will keep the discussion general.) 

At this point, depending of the parity of this rank, we proceed as follows:
\begin{itemize} 
\item if the rank happens to be even, we will perform  a \textit{section sum} of $Z$ and the product bundle $\Sigma_{h+1} \times \Sigma_9$ along sections of self--intersection $0$, i.e. a normal connected sum in which the boundaries of fibered tubular neighborhoods of the sections are identified with a base--preserving diffeomorphism, see \cite{BM13}; 
(Here and in what follows, the normal connected sum operation is not determined uniquely by the data above, but also by the isotopy class of the gluing map; all results we state, however, apply irrespective of those choices); 
\item if the rank happens to be odd, we will first perform a section sum of $Z$ and a standard building block $Q$ where $Q$ can be defined as the torus bundle over a surface of genus $9$ obtained by fiber summing the Kodaira--Thurston manifold -- a torus bundle over a torus -- with the product bundle $T^2 \times \Sigma_8$. (Note that $b_1(Q) = b_1(\Sigma_8) + 1$.) Next, we will proceed with a further section sum  with $\Sigma_{h} \times \Sigma_9$.
\end{itemize}

Either way, the resulting manifold will be a surface bundle $Z_{g,9}$ with fiber of genus $g := h+4$ and even first Betti number;  the monodromy representation of section sums is elucidated in \cite[Section 2.2]{BM13} and a little bookkeeping for the constructions above shows that in all cases $\textit{rk}_{\Z}H_1(\Sigma_g;\Z)_{\Pi_9} = 2g - 2s_0$ for some $s_0 \in [1,3]$. Moreover, by Novikov additivity (and by the choice of the building block using in the stabilizations, all having vanishing signature) the bundle $Z_{g,9}$ has signature $\sigma(Z_{g,9}) = \sigma(Z) > 0$. We can then increase the genus of the base by fiber sum with $\Sigma_{g} \times \Sigma_{c}$ to get a surface bundle $Z_{g,b}$ with base a surface of any genus $b \geq 9$. Again, Novikov additivity and a straightforward Mayer--Vietoris argument show that the signature and the coinvariant homology keep the properties discussed above.

As $g \geq 4$ is a free parameter in the construction above, we can make it large as wanted, in particular larger than $1+ 6s_0$, in which case its fundamental group is not the fundamental group of a Kodaira fibration by Proposition \ref{proposition:xi}. \end{proof}

We can make a minor variation to the construction above to illustrate another phenomenon, which shows that at times we can gain insight on the (non)existence of the structure of Kodaira fibration by reiterate application of Proposition \ref{proposition:xi} to finite covers.

\begin{proposition} \label{prop:reit} There exist a surface bundle $W_{g+1,b}$ with fiber of genus $g+1$ over a surface of genus $b \geq 9$ which satisfies the constraints of Proposition \ref{proposition:xi}, but admits an unramified cover $S_{g+1,5b-6}$ which fails them. Its fundamental group cannot occur as fundamental group of a Kodaira fibration. \end{proposition}

\begin{proof} The construction of our examples will build on the construction of Proposition \ref{prop:essezero} Consider the (say right--handed) trefoil knot $T(2,3)$. As detailed e.g. in \cite[Chapter 10]{Ro03} the trefoil is a fibered knot of genus $1$, whose exterior can thus be identified with the mapping torus of an automorphism $\phi: F \to F$ of a surface $F$ on genus $1$ with one boundary component $\partial F$, fixed pointwise by the monodromy $\phi$. With a suitable choice of a basis for $H_1(F;\Z) = \Z^2$, the homological monodromy is given by
\begin{equation} \label{eq:mono} \phi_{*} = \left( \begin{array}{cc}  1 & 1  \\  -1 & 0
\end{array} \right). \end{equation}

Next, perform a $0$--surgery of $S^3$ along the trefoil knot; the resulting closed $3$--manifold $N$ has $b_1(N) = 1$ and admits a torus fibration over $S^1$, whose fiber comes from capping off $F$ with a disk. The $4$--manifold $S^1 \times N$ is then a torus bundle over a torus, and the coinvariant homology of the fiber is trivial, as the homology monodromy representation is induced by that of the trefoil knot and coincides with that in (\ref{eq:mono}). Moreover, the signature of this bundle vanishes. Next, fiber sum with $T^2 \times \Sigma_{b-1}$ to get a torus bundle $M$ with base $\Sigma_{b}$; by the usual arguments, the coinvariant homology of the fiber and the signature vanish. The torus bundle $M$ admits a section of self--intersection $0$ (inherited from the dual of the trefoil knot in $N$). 

Take the surface bundle $Z_{g,b}$, that retains a section of self--intersection $0$, and perform a section sum with $M$. The resulting manifold $W_{g+1,b}$ is a surface bundle with fiber and surface genus as in the indexes, and $\sigma(W_{g+1,b})> 0$. The homology monodromy representation in $SL(2g+2,\Z)$  factors in block diagonal form, with a $2 \times 2$ block as in (\ref{eq:mono}) that comes from the homological monodromy of the trefoil knot. The coinvariant homology of the fiber satisfies $\textit{rk}_{\Z}H_1(\Sigma_{g+1};\Z)_{\Pi_b} = 2g -2s_0$ as the monodromy of the $2 \times 2$ block coming from $M$ increases by $1$  in the genus of the fiber but does not contribute to the coinvariant homology. Choose then $g$ such that $1+6s_0 < g + 1\leq 7+6s_0$; the latter inequality insures that $W_{g+1,b}$ does not violate the inequality of Proposition \ref{proposition:xi}, while the rationale for the first inequality will be apparent momentarily. Recall now that the homological monodromy of the trefoil knot is periodic of period $6$, as can be verified from 
($\ref{eq:mono}$) by computing explicitly $(\phi_{*})^{6} = I$. Then it is not too hard to see the existence of a $6$--fold (unramified) cyclic cover of $W_{g+1,b}$, that we will denote $S_{g+1,6b-5}$, induced by a $6$-fold cyclic cover $\Sigma_{6b-5} \to \Sigma_{b}$ of the base determined  by a normal subgroup $\Pi_{6b-5} \leq \Pi_{b}$. The covering map on $W_{g+1,b}$ is trivial on the fibers, hence the resulting cover $S_{g+1,6b-5}$ is again a surface bundle of fiber $\Sigma_{g+1}$ over the base $\Sigma_{6b-5}$, as recorded in the notation. Its coinvariant homology satisfies $\textit{rk}_{\Z}H_1(\Sigma_{g+1};\Z)_{\Pi_{6b-5}} \geq 2g+2 -2s_0$: there is in fact a gain, in rank, of at least two due to the periodicity of the $2 \times 2$ block in the homological monodromy representation in $SL(2g+2;\Z)$. (Further gain would depend on the monodromy of $Z$, as well as the choice of the gluing data in the section sums above.) Application of Proposition \ref{proposition:xi} entails that this manifold cannot carry the structure of Kodaira fibration, as having $\textit{rk}_{\Z}H_1(\Sigma_{g+1};\Z)_{\Pi_{6b-5}} \geq 2g+2 -2s_0$ would require to have at least $g+1 \leq 1 + 6s_0$, which violates the above choice of $g+1 > 1+6s_0$.  (If $\textit{rk}_{\Z}H_1(\Sigma_{g+1};\Z)_{\Pi_{6b-5}}$ happens to be odd, this would follow as well from simple parity reasons). As the finite cover of a Kodaira fibration must be a Kodaira fibration, $W_{g+1,b}$ does not carry the structure of Kodaira fibration either. Summarizing, reiterated use of Proposition \ref{proposition:xi} for finite covers may unveil further information. \end{proof}

\begin{remarks} \bn \item The surface bundles $Z_{g,b}$ and $W_{g+1,b}$ identified in the proofs of Propositions \ref{prop:essezero} and \ref{prop:reit} admit a symplectic structure by  \cite{Th76}. They fail however to admit a K\"ahler structure. We recap the argument:  if they did admit a K\"ahler structure, as their fundamental group  surjects onto $\Pi_b$ with finitely generated kernel, they would admit a holomorphic fibration over $\Sigma_b$. By the aforementioned results of Kotschick and Hillman, the fibers would be smooth with multiplicity one; as the signature is nonzero, these would be Kodaira fibrations, in violation of Proposition \ref{proposition:xi}. In our understanding, this type of obstruction to the existence of a K\"ahler structure on a symplectic manifold has not been spelled out  in this form before.

\item We conclude this note observing the rather annoying fact that all the discussion above has no bearing on the \textit{a priori} more basic question of whether a group as in ({\ref{eq:ses}}) can occur as fundamental group of some projective (or even K\"ahler) variety (compare Question (5) of \cite{Hi15}). In fact, we have no means to exclude that any of the groups we dismissed above could be the fundamental group of some projective variety hence, by standard arguments, of some algebraic surface $Y$. As the fundamental group of $Y$ would surject onto $\Pi_b$, such surface would admit a fibration with base $\Sigma_b$. By the above considerations, reversing the aforementioned result of \cite{Ko99,Hi00}, we can only infer that the fibration would be necessarily nonsmooth, $Y$ would not be aspherical, and its Euler characteristic would have to exceed $(2g-2)(2b-2)$. \en \end{remarks}
\begin{ack} The author would like to thank \.{I}nan\c{c} Baykur for useful comments. Also, I'm grateful for the anonymous referee for several suggestions that have improved the presentation of the paper. \end{ack}

%=========================================================

\end{document}